\documentclass[pdflatex,sn-mathphys]{sn-jnl}

\usepackage[utf8]{inputenc}

\usepackage{hyperref}
\usepackage{academicons}
\usepackage{xcolor}

\newcommand{\Prob}{\operatorname{\mathbb{P}}}
\newcommand{\E}{\operatorname{\mathbb{E}}}
\newcommand{\Var}{\operatorname{Var}}
\newcommand{\aut}{\operatorname{Aut}}

\jyear{2022}%

\theoremstyle{thmstyleone}%
\newtheorem{theorem}{Theorem}
\newtheorem{corollary}{Corollary}
\newtheorem{lemma}{Lemma}
\theoremstyle{thmstyletwo}%

\theoremstyle{thmstylethree}%

\begin{document}

\title[Automorphisms of random trees]{The distribution of the number of automorphisms of random trees}

\author[1]{\fnm{Christoffer} \sur{Olsson}}\email{christoffer.olsson@math.uu.se,\\ ORCID: 0000-0002-0864-0497}

\author[1]{\fnm{Stephan} \sur{Wagner}}\email{stephan.wagner@math.uu.se,\\ ORCID: 0000-0001-5533-2764}

\affil[1]{\orgdiv{Department of Mathematics}, \orgname{Uppsala Universitet}, \orgaddress{\country{Sweden}}}

\abstract{
We study the size of the automorphism group of two different types of random trees: Galton--Watson trees and rooted P\'olya trees. In both cases, we prove that it asymptotically follows a log-normal distribution and provide asymptotic formulas for mean and variance of the logarithm of the size of the automorphism group. While the proof for Galton--Watson trees mainly relies on probabilistic arguments and a general result on additive tree functionals, generating functions are used in the case of rooted P\'olya trees. We also show how to extend the results to some classes of unrooted trees.
}

\keywords{random tree, Galton--Watson tree, P\'olya tree, unrooted tree, automorphism group, moments, central limit theorem}

\maketitle

\let\svthefootnote\thefootnote
\let\thefootnote\relax\footnotetext{An earlier version of this paper has previously been published as the extended abstract \textit{Automorphisms of random trees} \cite{olsson_et_al:LIPIcs.AofA.2022.16}.}
\let\thefootnote\svthefootnote


\section{Introduction}
The automorphism group is a fundamental object associated with a graph as it encodes information about its symmetries. Furthermore, counting mathematical objects up to symmetry is a classical subject in combinatorics which naturally relates to the automorphism group. An example is the case of graphs, where the number of different labelings of a graph $G$ of order $n$ is given by $\frac{n!}{\vert\aut{G}\vert}$. In this paper we study properties of the automorphism groups associated with random trees, in particular Galton--Watson trees and P\'olya trees. We show that the size of the automorphism group follows a log-normal distribution with parameters depending on tree type. The size of the automorphism group has previously been studied in special cases of Galton--Watson trees: binary trees (expected values and limiting distribution: \cite{MR2582703}), labeled trees (limiting distribution: \cite{MR4472288} and expected value: \cite{LYUThesis}), binary and ternary trees (expected values: \cite{MR1170829} and \cite{MR1392319}). It has also been studied for some other types of trees than those considered here: specifically, random recursive trees (expected value: \cite{DMatthewsThesis}), and $d$-ary increasing trees (limiting distribution and moments: \cite{MR3984050}). We are primarily studying rooted trees but for some classes of trees we can extend the results to the unrooted case. The book \cite{DrmotaRandTrees} is a general reference to this introduction and the different types of random trees discussed in this paper.

Recall now that a \textit{Galton--Watson tree} is a growth model where we start with one vertex, the root, and the number of children it has is given by a (discrete) random variable $\xi$, supported on some subset of the non-negative integers that includes at least 0 and some number greater than 1. The tree grows by letting each of the vertices have children of their own according to the offspring distribution $\xi$, independently of all other vertices. Different distributions for $\xi$ give rise to different types of Galton--Watson trees. We are especially interested in the case of \textit{critical} Galton--Watson trees, for which $\E \xi = 1$, as well as \textit{conditioned} Galton--Watson trees where we condition on the size of the tree, i.e., we pick one of all possible Galton--Watson trees on $n$ vertices at random. A related notion is that of the \textit{size-biased} Galton--Watson tree, which has two different types of vertices. The \textit{normal} vertices have the same offspring distribution $\xi$ as before, while the \textit{special} vertices get offspring according to the size-biased distribution $\hat{\xi}$ defined by $\Prob(\hat{\xi}=k) = k\Prob(\xi=k)$. We start the growth process with the root being special, and for each special vertex we choose exactly one of its children, uniformly at random, to be special as well. This means that the size-biased Galton--Watson tree has an infinite \textit{spine} of special vertices, with non-biased unconditioned Galton--Watson trees attached to it. Conditioned Galton--Watson trees are closely connected to, and a special case of, \textit{simply generated families} of trees (or \textit{simple trees}) which are defined in terms of generating functions. For a sequence of non-negative numbers $\{w_k\}$ define
\begin{equation*}
    \Phi(z) = \sum_{k\geq 0} w_k x^k 
\end{equation*}
to be its \textit{weight generating function}. Then the generating function for the class of trees associated with $\{w_k\}$,
\begin{equation*}
    T(x) = \sum_{T\in\mathcal{T}} w(T) x^{\vert T\vert} 
\end{equation*}
is defined by the functional equation
\begin{equation}
\label{eq:SimpleFuncEq}
    T(x) = x\Phi(T(x)) .
\end{equation}
The number $w(T)$ is called the \textit{weight} of the tree $T$. Under the (mild) assumption that there exists a positive $\tau$ within the radius of convergence of $\Phi(z)$ such that
\begin{equation*}
    \Phi(\tau) = \tau\Phi'(\tau) < \infty, 
\end{equation*}
we can find $\rho=\frac{\tau}{\Phi(\tau)}$ such that $T(x)$ has the singular expansion
\begin{equation}
\label{eq:GWSingExp}
    T(x) = \tau - c_1\sqrt{1-\frac{x}{\rho}} + \sum_{k\geq 2} (-1)^k c_k \left(1-\frac{x}{\rho}\right)^{\frac{k}{2}} ,
\end{equation}
for constants $c_k$ that can be calculated. Through the process of singularity analysis, this implies that the total weight of all trees of size $n$ is asymptotic to
\begin{equation*}
     C n^{-3/2}\rho^{-n} .
\end{equation*}
We take the probability of picking a given simple tree $S$ of size $n$ to be
\begin{equation}
\label{eq:ProbSimple}
    \frac{w(S)}{\sum_{\vert T\vert=n} w(T)} .
\end{equation}

We can see Galton--Watson trees and simple trees as two sides of the same coin, one being probabilistic and the other being combinatorial, where Galton--Watson trees correspond to simply generated trees with weights $w_k$ adding up to 1. In this context, the numbers $w_k$ correspond to the probability of a vertex having $k$ children, $w(T)$ is the probability of obtaining $T$ through the Galton--Watson growth process and \eqref{eq:ProbSimple} is the probability when we condition on the size of the tree. In fact, if we can find a $\tau$ as above, we can always assume that our trees, whether they are conditioned Galton--Watson or simply generated ones, are critical Galton--Watson trees as long as we can perform slight modifications (which will not affect the probabilities of individual trees) to the offspring distribution. Then, the critical Galton--Watson trees are those simple trees having their dominant singularity at $\rho=1$, so that the discussion above indicates that the probability of an (unconditional) Galton--Watson tree having size $n$ decays like $Cn^{-3/2}$. Examples of Galton--Watson (and simply generated) trees are plane trees, labeled trees, $d$-ary trees, etc. 

P\'olya trees are unordered, unlabeled trees which can be either rooted or unrooted. Rooted Pólya trees have many properties similar to Galton--Watson trees, but they do not satisfy the definition and cannot be interpreted as growth processes so we will need other methods to deal with them. They can be characterized by their generating function $P(x) = \sum_{T\in\mathcal{P}} x^{\vert T\vert}$, which satisfies
\begin{equation}
\label{eq:PolyaOrdFuncEq}
    P(x) = x\exp\left(\sum_{k=1}^\infty \frac{P(x^k)}{k} \right) .
\end{equation}
The number of such trees of size $n$ is asymptotic to $A n^{-3/2}\rho_p^{-n}$, where $\rho_p= 0.33832\ldots$ is the dominant singularity of $P(x)$ and $A$ is a constant. For this singularity, we have $P(\rho_p)=1$.

A classical result gives a bijection between Pólya trees and the union of unrooted unlabeled trees together with pairs of distinct Pólya trees. The bijection translates into the functional equation 
\begin{equation}
\label{eq:ClassicalUnrooted}
    U(x) = P(x) - \frac{1}{2} P(x)^2 + \frac{1}{2}P(x^2)
\end{equation}
that describes the generating function for unrooted trees $U(x)$ in terms of $P(x)$. The number of unrooted Pólya trees of size $n$ is asymptotic to $B n^{-5/2}\rho_p^{-n}$ for a constant $B$.

We use $\mathcal{T}$ to denote Galton--Watson trees, $\mathcal{T}_n$ to denote conditioned Galton--Watson trees on $n$ vertices and $\hat{\mathcal{T}}$ to denote size-biased trees. Similarly, we use $T$, $T_n$ and $\hat{T}$ to denote specific realizations of the respective trees. Furthermore, we will use $\mathcal{P}$ and $\mathcal{P}_n$ to denote rooted P\'olya trees as well as P\'olya trees of size $n$, respectively and, sometimes, $\mathcal{U}$ and $\mathcal{U}_n$ in the case of unrooted trees. We let $\textrm{mult}(B)$ be the number of occurrences of a particular tree $B$ as root branches of some other tree. Note that the isomorphism classes of Galton--Waton trees are rooted Pólya trees. In addition to using $w(T)$ for the weight of a simple tree, we will use $W(B)$ to denote the weight of the entire isomorphism class $B$.

\subsection{Results}
In this paper, we will show asymptotic normality of $\log\vert\aut{\mathcal{T}_n}\vert$, for various classes of random trees. This implies asymptotic log-normality of $\vert\aut{\mathcal{T}_n}\vert$.
We prove the following theorem on the automorphism group of Galton--Watson trees.
\begin{theorem}
\label{thm:MainThmGW}
 Let $\mathcal{T}_n$ be a conditioned Galton--Watson tree of order $n$ with offspring distribution $\xi$, where $\E\xi=1$, $0<\Var\xi<\infty$ and $\E\xi^5<\infty$. Then there exist constants $\mu$ and $\sigma^2\geq 0$, depending on $\mathcal{T}$, such that
 \begin{equation*}
     \frac{\log\vert\aut{\mathcal{T}_n}\vert-\mu n}{\sqrt{n}} \xrightarrow[]{d} \mathrm{N}(0, \sigma^2).
 \end{equation*}
\end{theorem}
The condition on $\E\xi^5$ is needed for technical purposes and is valid for combinatorially significant examples such as labeled trees, plane trees and $d$-ary trees. The exponent $5$ is probably not best possible, but required to apply the general result on additive functionals that our proof is based on.

The mean constant $\mu$ and even more so the variance constant $\sigma^2$ do not seem easy to compute numerically in general. We show how to derive the numerical values for some classes of trees, namely labeled trees as well as general Galton--Watson trees with bounded degrees. Numerical estimates for some types of trees can be found in Table \ref{table:NumConst}. 

\begin{table*}[b] 
    \centering
    \begin{tabular}{|c|c|c|} \hline
        Class of tree & $\mu$ & $\sigma^2$  \\ \hline
        Labeled trees & 0.0522901 & 0.0394984 \\
        Full binary trees & 0.0939359 & 0.0252103 \\
        Pruned binary trees & 0.0145850 & 0.0084835 \\ \hline
        Pólya trees & 0.1373423 & 0.1967696  \\\hline
    \end{tabular}
    \caption{Numerical estimates of the mean and variance constants for some types of trees}
    \label{table:NumConst}
\end{table*}

Note that it is unclear what an unrooted version of a Galton--Watson tree is in general so we cannot expect an unrooted version of Theorem \ref{thm:MainThmGW}, but in the case of labeled trees, the result for rooted trees translates to the case of unrooted trees as well. 
\begin{theorem}
Let $\mathcal{T}_n$ be a uniformly random unrooted labeled tree of size $n$. Then, $\E(\log\vert\aut{\mathcal{T}_n}\vert) = \mu n + O(1)$ and $\Var(\log\vert\aut{\mathcal{T}_n}\vert) = \sigma^2 n + O(1)$, with $\mu=0.0522901\ldots$ and $\sigma^2=0.0394984\ldots$. Furthermore, we have
 \begin{equation*}
     \frac{\log\vert\aut{\mathcal{T}_n}\vert-\mu n}{\sqrt{n}} \xrightarrow[]{d} \mathrm{N}(0, \sigma^2).
 \end{equation*}
\end{theorem}

We can also prove asymptotic log-normality for the size of the automorphism group of Pólya trees.
\begin{theorem}
\label{thm:MainThmP}
Let $\mathcal{P}_n$ be a uniformly random P\'olya tree of order $n$, rooted or unrooted. Then, $\E(\log\vert\aut{\mathcal{P}_n}\vert) = \mu n + O(1)$ and $\Var(\log\vert\aut{\mathcal{P}_n}\vert) = \sigma^2 n + O(1)$, with $\mu=0.1373423\ldots$ and $\sigma^2=0.1967696\ldots$. Furthermore, we have
 \begin{equation*}
     \frac{\log\vert\aut{\mathcal{P}_n}\vert-\mu n}{\sqrt{n}} \xrightarrow[]{d} \mathrm{N}(0, \sigma^2).
 \end{equation*}
\end{theorem}

The proofs of Theorem \ref{thm:MainThmGW} and Theorem \ref{thm:MainThmP} rely at their cores on the same idea of approximating the additive functionals by simpler ones, but they are fairly different at a glance. We give some preliminary results in Section \ref{sec:prel}. We then prove Theorem \ref{thm:MainThmGW} in Section \ref{sec:GW} and Theorem \ref{thm:MainThmP} for rooted trees in Section \ref{sec:RootedPolya}. The results for unrooted trees are proved in Section \ref{sec:Unrooted}.

\section{Preliminaries}
\label{sec:prel}

For any rooted tree $T$, we have a recursive formula for the size of its automorphism group. Let $T_1, T_2, \ldots, T_k$ be its root branches up to isomorphism, having multiplicities $m_1, m_2, \ldots, m_k$, respectively. Then we have
\begin{equation}
\label{eq:AutRecursive}
    \vert\aut{T}\vert = \prod_{i=1}^k m_i! \vert\aut{T_i}\vert^{m_i} , 
\end{equation}
derived from the fact that the automorphism group of a rooted tree is obtained from symmetric groups by iterated direct and wreath products (see \cite{MR1373683}, Proposition 1.15). In other words, the tree is invariant under the automorphisms of each of the root branches as well as under permutation of isomorphic branches. By taking logarithms, we find that $\log\vert\aut{T}\vert$ is an \textit{additive functional} of the tree, which is a real-valued function $F(T)$ that satisfies
\begin{equation*}
    F(T) = f(T) + \sum_{i=1}^r F(S_i) , 
\end{equation*}
where we sum over the $r$ (possibly isomorphic) root branches $S_1,S_2,\ldots,S_r$ and $f(T)$ is a function called the \textit{toll} of the additive functional. From \eqref{eq:AutRecursive}, we see that $\log\vert\aut{T}\vert$ has the equivalent form
\begin{equation*}
    F(T) = f(T) + \sum_{i=1}^k m_i F(T_i) , 
\end{equation*}
where the sum is over root branches up to isomorphism. In this case the toll function is $f(T) = \sum \log(m_i!)$. 

Limit theorems for additive functionals have been proven for various classes of random trees under different conditions, see \cite{MR3311217, MR3318048, MR3432572, MR3984050, MR3690263, almostLocal_2020}. In the case of Galton--Watson trees, we will specifically make use of a general result on \emph{almost local} additive functionals due to Ralaivaosaona, \v{S}ileikis and the second author \cite{almostLocal_2020}, which is in turn based on earlier work by Janson \cite{MR3432572}. Intuitively, ``almost local'' means that looking at the first $M$ levels of the tree gives us substantial (albeit not perfect) information about the value of the toll function at the root. We will let $\mathcal{T}^{(M)}$ denote the restriction of a Galton--Watson tree to its first $M$ levels, where the root is at level 0, with similar definitions for the other classes of trees. The theorem we will use is the following. 

\begin{theorem}[\cite{almostLocal_2020}]
\label{thm:almostlocal}
 Let $\mathcal{T}_n$ be a conditioned Galton--Watson tree of order $n$ with offspring distribution $\xi$, with $\E\xi=1$ and $0<\sigma^2:=\Var\xi<\infty$. Assume further that $\E\xi^{2\alpha + 1}<\infty$ for some integer $\alpha\geq0$. Consider a functional $F$ of finite rooted ordered trees with the property that 
 \begin{equation*}
     f(T) = O(\mathrm{deg}(T)^\alpha), 
 \end{equation*}
 where $f$ is the toll function associated with the functional.
 
Furthermore, assume that there exists a sequence $(p_M)_{M\geq1}$ of positive numbers with $p_M\to 0$ as $M\to\infty$, such that
\begin{itemize}
     \item for every integer $M\geq 1$, 
     \begin{equation*}
         \E\left\vert f(\hat{\mathcal{T}}^{(M)})-\E\left(f(\hat{\mathcal{T}}^{(N)})\vert\hat{\mathcal{T}}^{(M)}\right)\right\vert \leq p_M , 
     \end{equation*}
     for all $N\geq M$, 
     \item there is a sequence of positive integers $(M_n)_{n\geq1}$ such that  for large enough $n$, 
     \begin{equation*}
         \E\vert f(\mathcal{T}_n) - f(\mathcal{T}_n^{(M)})\vert \leq p_{M_n}.
     \end{equation*}
 \end{itemize}
 If $a_n = n^{-1/2}(n^{\max\{\alpha, 1\}}p_{M_n}+ M_n^2)$ satisfies
     \begin{equation*}
         \lim_{n\to\infty} a_n=0, \textrm{ and }\sum_{n=1}^\infty \frac{a_n}{n} < \infty , 
     \end{equation*}
     then
     \begin{equation*}
         \frac{F(\mathcal{T}_n) - \mu n}{\sqrt{n}} \xrightarrow[]{d} N(0, \gamma^2) , 
     \end{equation*}
     where $\mu = \E f(\mathcal{T})$ and $0\leq \gamma^2<\infty$.
\end{theorem}
The proof shows that the result still holds if we replace $(F(\mathcal{T}_n) -\mu n)/\sqrt{n}$ by $(F(\mathcal{T}_n) - \E F(\mathcal{T}_n))/\sqrt{n}$.

To prove the result for Pólya trees we will instead rely on generating functions. We can define the generating function of $F(\mathcal{P}_n) = \log\vert\aut{\mathcal{P}_n}\vert$ to be
\begin{equation}
\label{eq:GenFuncPolya}
    P(x, t) = \sum_{T\in\mathcal{P}} e^{t\log\vert\aut{T}\vert} x^{\vert T\vert} = \sum_{T\in\mathcal{P}} \vert\aut{T}\vert^{t} x^{\vert T\vert} .
\end{equation}
Note that $P(x, 0) = P(x)$. We can now derive a functional equation analogous to \eqref{eq:PolyaOrdFuncEq} as follows. We have the symbolic decomposition
\begin{equation*}
    \mathcal{P} = \bullet \times \bigotimes_{T\in\mathcal{P}} (\emptyset \uplus \{T\} \uplus \{T, T\} \uplus \ldots )  , 
\end{equation*}
reflecting the fact that a P\'olya tree consists of a tree and a multiset of branches. Taking automorphisms into account, this translates to
\begin{equation*}
    P(x, t) =  x \prod_{T\in\mathcal{P}}\left(\sum_{n=0}^\infty x^{n\vert T\vert} n!^t \vert\aut{T}\vert^{nt}\right)   , 
\end{equation*}
by general principles for generating functions. We can manipulate this as follows:
\begin{multline*}
P(x, t) =  x \exp\left(\sum_{T\in\mathcal{P}} \log \left(\sum_{n=0}^\infty x^{n\vert T\vert} n!^{t} \vert\aut{T}\vert^{nt}\right)\right) \\
= x \exp\left(\sum_{T\in\mathcal{P}} \sum_{k=1}^\infty \frac{(-1)^{k-1}}{k} \left(\sum_{n=1}^\infty x^{n\vert T\vert} n!^{t} \vert\aut{T}\vert^{nt}\right)^k \right).
\end{multline*}
The sum in the exponent can be rewritten as
\begin{equation*}
    \sum_{T\in\mathcal{P}} \sum_{k=1}^\infty \frac{(-1)^{k-1}}{k} \sum_{\substack{\lambda_1+\lambda_2\\+\cdots = k}} \binom{k}{\lambda_1, \lambda_2, \ldots} \prod_{n=1}^\infty \big(x^{n\vert T\vert} n!^{t} \vert\aut{T}\vert^{nt}\big)^{\lambda_n}
\end{equation*}
We now write integer partitions as sequences $\lambda = (\lambda_1, \lambda_2, \ldots)$, where $\lambda_i$ is the number of $i$'s in the partition. The total number of summands is denoted by $\vert\lambda\vert = \lambda_1+\lambda_2+\ldots$, and we write $\lambda \vdash j$ to denote that $\lambda$ is a partition of $j$, i.e. $j = \lambda_1 + 2\lambda_2+3\lambda_3 + \ldots$. Further manipulations give
\begin{multline*}
    \sum_{T\in\mathcal{P}} \sum_{k=1}^\infty \frac{(-1)^{k-1}}{k} \sum_{j=1}^\infty \sum_{\substack{\lambda_1+\lambda_2+\ldots = k\\\lambda_1+2\lambda_2+\ldots = j}} \binom{k}{\lambda_1, \lambda_2, \ldots}  x^{j\vert T\vert}  \vert\aut{T}\vert^{jt}\prod_{n=1}^\infty n!^{\lambda_n t } 
    \\
    = 
    \sum_{j=1}^\infty \sum_{\lambda\vdash j}  \frac{(-1)^{\vert\lambda\vert-1}}{\vert\lambda\vert}  \binom{\vert\lambda\vert}{\lambda_1, \lambda_2, \ldots} \left(\prod_{n=1}^{\infty} n!^{\lambda_n t}\right) \sum_{T\in\mathcal{P}} x^{j\vert T\vert}  \vert\aut{T}\vert^{jt} 
    \\
    = 
    \sum_{j=1}^\infty \sum_{\lambda\vdash j}  \frac{(-1)^{\vert\lambda\vert-1}}{\vert\lambda\vert}  \binom{\vert\lambda\vert}{\lambda_1, \lambda_2, \ldots} \left(\prod_{n=1}^{\infty} n!^{\lambda_n t}\right)  P(x^j, jt) 
    .
\end{multline*}
For convenience, we can define
\begin{equation*}
    c(j, t) = j \sum_{\lambda\vdash j}  \frac{(-1)^{\vert\lambda\vert-1}}{\vert\lambda\vert}  \binom{\vert\lambda\vert}{\lambda_1, \lambda_2, \ldots} \left(\prod_{n=1}^{\infty} n!^{\lambda_n t}\right) , 
\end{equation*}
and arrive at the functional equation 
\begin{equation}\label{eq:PolyaEqAut}
    P(x, t) = x \exp \left( P(x, t) + \sum_{j=2}^\infty \frac{c(j, t)}{j}P(x^j, jt)\right) .
\end{equation}
Note that $c(j, 0) = 1$, so that we recover the functional equation \eqref{eq:PolyaOrdFuncEq} if we set $t=0$.

\section{The automorphism group of Galton--Watson trees}
\label{sec:GW}
As indicated in the previous section, we will show that $\log\vert\aut{\mathcal{T}_n}\vert$ is in fact an almost local additive functional. This will let us apply Theorem \ref{thm:almostlocal} to prove that it converges in distribution to a normal random variable.

\subsection{Galton--Watson trees isomorphic up to a certain level}
In applying Theorem \ref{thm:almostlocal}, we are led to consider the probability that two Galton--Watson trees are of height $\geq M$ and isomorphic. We use $\mathcal{C}$ to denote the set of isomorphism classes of Galton--Watson trees as well as $\mathcal{C}^M$ to denote the set of isomorphism classes of trees of height $M$ (i.e., trees that have $M+1$ generations). The definitions extend to conditioned Galton--Watson trees as $\mathcal{C}_n$ and $\mathcal{C}_n^M$, respectively. We start with the following lemma.

\begin{lemma}
\label{lemma:cut-offUncondGW}
There exists some constant $0<c<1$ such that
\begin{equation*}
    \Prob(\mathcal{T}^{(M)} \textrm{ belongs to } C) \leq c^{M} , 
\end{equation*}
uniformly for all isomorphism classes $C\in\mathcal{C}^M$.
\end{lemma}

\begin{proof}
We say that a level $L$ of a tree $T$ agrees with $C$ if it has the correct number of vertices and the offsprings $\xi_1, \xi_2, \ldots, \xi_l$ agree with the offsprings of the same level in $C$, up to permutation. Let $L_1, L_2, \ldots$ denote the levels of the Galton--Watson tree $\mathcal{T}$. Then the probability is bounded by
\begin{equation}
\label{eq:Lemmacut-offUncondGW}
    \Prob(\mathcal{T}^{(M)} \textrm{ belongs to } C) \leq \prod_{i=0}^{M-1} \Prob(L_i \textrm{ agrees with }C \vert L_1, L_2, \ldots, L_{i-1}) , 
\end{equation}
where we note that, by truncation, the $M$-th level will always agree with $C$, as long as the previous ones do. We can bound each factor in \eqref{eq:Lemmacut-offUncondGW} by the probability of the level having the correct number of leaves, conditioned on the previous levels. This random variable follows a binomial distribution with probability $p=\Prob(\xi = 0)$. It is therefore sufficient to prove a bound $0<c<1$ (uniform in both $l$ and $k$) on the probability that a binomial variable $X_l\sim\mathrm{Bin}(l, p)$ takes a specific value $k$.

We can in fact bound $X_l$ in terms of $p$, since if we write $X_l$ as a sum of Bernoulli variables $X_l = Y_1 + Y_2 + \ldots+ Y_l$ we have
\begin{multline*}
    \Prob(Y_1 + Y_2 + \ldots + Y_l = k) = \sum_{r=0}^1 \Prob(Y_1+Y_2+\ldots+Y_{l-1}=k-r)\Prob(Y_l = r) \\
    \leq \sum_{r=0}^1 \Prob(Y_1+Y_2+\ldots+Y_{l-1}=k-r)\max_{y\in\{0,1\}}\Prob(Y_l = y) \leq  \max\{p, 1-p\} .
\end{multline*}
We can thus take $c=\max \{p, 1-p\}$ as a uniform bound for all levels, and now \eqref{eq:Lemmacut-offUncondGW} gives the result. 
\end{proof}

We now see that for two independent trees $\mathcal{T}_1, \mathcal{T}_2$ we have
\begin{multline}
\label{mltline:isoTrees}
    \Prob(\mathcal{T}_1^{(M)}, \mathcal{T}_2^{(M)} \textrm{ iso. and of height}\geq M ) = \sum_{C\in\mathcal{C}^M} \Prob(\mathcal{T}^{(M)} \textrm{ belongs to } C)^2 \\
    \leq \max_{C\in\mathcal{C}^M}\{\Prob(\mathcal{T}^{(M)} \textrm{ belongs to } C)\} \sum_{C\in\mathcal{C}^M} \Prob(\mathcal{T}^{(M)} \textrm{ belongs to } C) \\
    = \max_{C\in\mathcal{C}^M}\{\Prob(\mathcal{T}^{(M)} \textrm{ belongs to } C)\} .
\end{multline}
Combining this with Lemma \ref{lemma:cut-offUncondGW}, we get the following corollary.

\begin{corollary}
\label{cor:uncondGWiso}
Let $\mathcal{T}_1, \mathcal{T}_2$ be two independent Galton--Watson trees. There exists some constant $0<c<1$ such that
\begin{equation*}
     \Prob(\mathcal{T}_1^{(M)}, \mathcal{T}_2^{(M)} \textrm{ isomorphic and of height}\geq M ) \leq c^M .
\end{equation*}
\end{corollary}

In fact, the argument in \eqref{mltline:isoTrees} also works when one of the trees is the size-biased tree $\hat{\mathcal{T}}$, which lets us bound the probability that a Galton--Watson tree and the size-biased tree are isomorphic up to level $M$ in terms of the maximum probability that the Galton--Watson tree belongs to a specific isomorphism class. This gives another corollary, which we will need later on.

\begin{corollary}
\label{cor:sizeBiasIso}
Let $\mathcal{T}$ be a Galton--Watson tree and $\hat{\mathcal{T}}$ be the size-biased tree, assumed to be independent of $\mathcal{T}$. There exists some constant $0<c<1$ such that
\begin{equation*}
     \Prob(\mathcal{T}^{(M)}, \hat{\mathcal{T}}^{(M)} \textrm{ isomorphic and of height}\geq M ) \leq c^M .
\end{equation*}
\end{corollary}

We can obtain similar bounds on the probability that two conditioned Galton--Watson trees are isomorphic up to level $M$. We start by extending Lemma \ref{lemma:cut-offUncondGW} to the conditioned case.

\begin{lemma}
\label{lemma:cut-offCondGW}
Let $\mathcal{T}_n$ be a conditioned Galton--Watson tree of size $n$. There exists some constant $0<c<1$ such that
\begin{equation*}
    \Prob(\mathcal{T}_n^{(M)} \textrm{ belongs to } C) = O\left(n^{\frac{5}{2}} c^{M} \right), 
\end{equation*}
uniformly for all isomorphism classes $C\in\mathcal{C}_{n}^{M}$.
\end{lemma}

\begin{proof}
Order the offsprings $\xi_1, \xi_2, \ldots$ of $T_n$ in breadth-first order and consider the sums
\begin{equation*}
    S_m = \sum_{i=1}^m (\xi_i-1) \quad \textrm{for } 1\leq m\leq n . 
\end{equation*}
In each step, $1\leq i\leq m$, we are deleting 1 for the current vertex while adding the number of children it has. For a conditioned Galton--Watson tree of size $n$, we necessarily have
\begin{align*}
    S_m >& -1 \quad \textrm{for } 1\leq m < n, \\
    S_n =& -1 , 
\end{align*}
since we are adding 1 for all vertices except the root, but deleting 1 for all vertices including the root. Using this, we can formulate the probability we seek to bound in the following way. 
\begin{equation*}
    \Prob(T_n^{(M)} \textrm{ belongs to } C) 
    = \frac{\Prob(\{T' \textrm{ belongs to } C\}\cap \{S_1, S_2, \ldots, S_{n-1}>-1, S_n=-1\})}{\Prob(S_1, S_2, \ldots, S_{n-1}>-1, S_n=-1)}, 
\end{equation*}
where $T'$ is a Galton--Watson tree with offsprings $\xi_1, \xi_2, \ldots, \xi_k$, and $k$ is the number of vertices of each tree in $C$ excluding the last level (since we truncate at level $M$, the number of children the vertices on this level have is of no interest to us). Since the trees in $C$ are isomorphic, they will all have the same number of vertices.

Let $l_M$ be the number of vertices at the last level of each tree in $C$ (again, equal due to isomorphism). Then we have
\begin{equation*}
    \sum_{i=1}^n (\xi_i-1) = \sum_{i=1}^{k} (\xi_i-1) + \sum_{i=k+1}^n (\xi_i-1) = l_M - 1 + \sum_{i=k+1}^{n} (\xi_i-1).
\end{equation*}
By the conditions set on $S_m$, we draw the conclusion that
\begin{align*}
 S_m' :=& \sum_{i=k+1}^{k+m} (\xi_i-1) > -l_M \quad \textrm{for } 1\leq m < n-k, \\
 S_{n-k}' :=& \sum_{i=k+1}^{n} (\xi_i-1) = -l_M .
\end{align*}
By independence, we now have
\begin{multline*}
    \frac{\Prob(\{T' \textrm{ belongs to } C\}\cap \{S_1, S_2, \ldots, S_{n-1}>-1, S_n=-1\})}{\Prob(S_1, S_2, \ldots, S_{n-1}>-1, S_n=-1)} \\
    = \frac{\Prob(T' \textrm{ belongs to } C) \Prob(S_1', S_2', \ldots, S_{n-k-1}'>-l_M, S_{n-k}'=-l_M)}{\Prob(S_1, S_2, \ldots, S_{n-1}>-1, S_n=-1)} , 
\end{multline*}
and using the cycle lemma we find that this equals
\begin{equation*}
    \frac{\frac{l_M}{n-k}\Prob(S_{n-k}'=-l_M)}{\frac{1}{n}\Prob(S_n=-1)} \Prob(T' \textrm{ belongs to } C) .
\end{equation*}
The probability $\Prob(S_{n-k}'=-l_M)$ is bounded by $1$, and $S_n$ satisfies a local limit theorem. If we also bound $l_M\leq n$ as well as $n-k\geq 1$ ($k$ is the number of vertices up to level $M-1$, and by definition there must be at least one vertex at level $M$) and use Lemma \ref{lemma:cut-offUncondGW} (note that $\mathcal{C}_{n, M}$ is a subset of $\mathcal{C}^M$), we arrive at
\begin{equation*}
    \Prob(T_n^{(M)} \textrm{ belongs to } C) = O\left(n^{\frac{5}{2}} c^M\right) , 
\end{equation*}
which is what we wanted to prove.

\end{proof}

Furthermore, using calculations similar to \eqref{mltline:isoTrees}, we obtain the following corollary.
\begin{corollary}
Let $\mathcal{T}_{n_1}, \mathcal{T}_{n_2}$ be two independent conditioned Galton--Watson trees. There exists some constant $0<c<1$ such that
\begin{equation*}
     \Prob(\mathcal{T}_{n_1}^{(M)}, \mathcal{T}_{n_2}^{(M)} \textrm{ isomorphic and of height}\geq M ) =  O\left(n^{\frac{5}{2}} c^{M} \right) ,
\end{equation*}
where we can take $n=\min\{n_1,n_2\}$.
\end{corollary}

We are now ready to apply the central limit theorem for additive functionals.

\subsection{Applying the CLT for almost local additive functionals}
By Stirling's approximation, we can bound $f(T) \leq \log\mathrm{deg}(T)! = O(\mathrm{deg}(T)^ {1+\epsilon})$ for any $\epsilon>0$, so that the functional satisfies the degree condition of Theorem \ref{thm:almostlocal} with $\alpha = 2$. For the expectations, there are two conditions to check, one for the size-biased Galton--Watson tree and one for the conditioned Galton--Watson tree, and in each case the difference inside the expectation can only be non-zero if (at least) two branches are isomorphic up to level $M$ but non-isomorphic when we take all levels into account. We can therefore reduce the problem to studying trees that are isomorphic up to the $M$-th level. 

We note that if $l$ root branches are isomorphic up to level $M$, this contributes at most $\log(l!)\leq \binom{l}{2}$ to the difference inside the expectation.
Therefore, the contribution of a random tree can be bounded by the sum of indicators
\begin{equation*}
    \sum_{T_i, T_j\textrm{ root branches}} I(T_i^{(M)}, T_j^{(M)} \textrm{ isomorphic and of height}\geq M ) , 
\end{equation*}
where we sum over distinct branches. We can thus bound the expectation $\E\vert f(\mathcal{T}_n) - f(\mathcal{T}_n^{(M)})\vert$ for the conditioned Galton--Watson tree by
\begin{equation*}
    \E\left(\sum_{\substack{\mathcal{T}_i, \mathcal{T}_j\\\textrm{ root branches}}} I(\mathcal{T}_i^{(M)}, \mathcal{T}_j^{(M)} \textrm{ are iso.~with height}\geq M)\right) .
\end{equation*}
This can, in turn, be bounded by
\begin{multline*}
    \sum_{k\geq 2} \Prob(\textrm{deg}(\mathcal{T}_n) = k)  \sum_{n_1, n_2}  \Prob(\vert\mathcal{T}_i\vert=n_1\vert\textrm{deg}(\mathcal{T}_n) = k)\Prob(\vert\mathcal{T}_j\vert=n_2\vert\textrm{deg}(\mathcal{T}_n) = k) \\
    \cdot \binom{k}{2} \E\left(   I(\mathcal{T}_i^{(M)}, \mathcal{T}_j^{(M)}\textrm{ iso. with height}\geq M)  \bigg\vert \vert\mathcal{T}_i\vert=n_1, \vert\mathcal{T}_j\vert=n_2 \right) \\
    = O\left(\sum_{k\geq 2} \Prob(\textrm{deg}(\mathcal{T}_n) = k) \binom{k}{2} n^{\frac{5}{2}} c^{M} \right) = O\left(n^{\frac{5}{2}} c^{M} \sum_{k\geq 2} k\Prob(\xi = k) \binom{k}{2}\right) 
\end{multline*}
where we use the law of total expectation and the fact that $\Prob(\textrm{deg}(\mathcal{T}_n) = k) \leq c k\Prob(\xi=k)$ for all $k$ and $n$, where $c$ is constant \cite[(2.7)]{janson05}. By assumptions on the moments of the offspring distribution, this expression is $O(n^{\frac{5}{2}} c^{M})$.

The difference $\vert f(\hat{\mathcal{T}}^{(M)})-\E(f(\hat{\mathcal{T}}^{(N)})\vert\hat{\mathcal{T}}^{(M)})\vert$ must also be zero unless some branches are isomorphic up to level $M$, and reasoning similar to above lets us rewrite its expectation in the following way.
\begin{multline*}
    \sum_{k\geq2} kP(\xi=k) 
     \Bigg( \E\Big(  \sum_{\substack{\mathcal{T}_i, \mathcal{T}_j\textrm{ non-special}\\\textrm{ root branches}}} I(\mathcal{T}_i^{(M)}, \mathcal{T}_j^{(M)} \textrm{ iso. with height}\geq M) \Big) \\
    + \E\Big(  \sum_{\substack{\mathcal{T}\textrm{ non-special root branch}\\ \hat{\mathcal{T}} \textrm{ special root branch}}} I(\mathcal{T}^{(M)}, \hat{\mathcal{T}}^{(M)} \textrm{ iso. with height}\geq M) \Big) \Bigg) ,
\end{multline*}
which is equal to
\begin{multline*}
    \sum_{k\geq3} kP(\xi=k) \binom{k-1}{2} \Prob(\mathcal{T}_1^{(M)}, \mathcal{T}_2^{(M)} \textrm{ iso. and of height}\geq M ) \\
    + \sum_{k\geq2} kP(\xi=k) (k-1) \Prob(\mathcal{T}^{(M)}, \hat{\mathcal{T}}^{(M)} \textrm{ iso. and of height}\geq M ) 
    = O(c^M)  
\end{multline*}
by Corollaries \ref{cor:uncondGWiso} and \ref{cor:sizeBiasIso} (the constant $c$ is the same for both of these corollaries since they both rely on Lemma \ref{lemma:cut-offUncondGW}) as well as assumptions on moments of the offspring distribution.

We now set $p_M = Kc_1^M$, for $c<c_1<1$ and some suitable constant $K$, as well as $M_n = A\log n$, for some positive constant $A$ that is large enough to make $n^{5/2}c^{M_n} \leq c_1^{M_n}$ for all $n$ and $A\log c_1 < -3/2$. Then, the expectations mentioned in Theorem \ref{thm:almostlocal} are bounded by $p_M$ and $p_{M_n}$, respectively. Furthermore, the sequence $a_n$ goes to $0$ and satisfies $\sum a_n/n <\infty$. Thus, we can apply Theorem~\ref{thm:almostlocal} to show that $\log\vert\aut{\mathcal{T}_n}\vert$ is asymptotically normal, which completes the proof of Theorem~\ref{thm:MainThmGW}.

\subsection{Mean and variance for some classes of trees}
In general, calculating the mean and variance constants for Galton--Watson trees seems to be a difficult feat, but we show how to do it in the special cases of labeled trees as well as Galton--Watson trees with bounded degrees. In both cases we view the trees as simply generated and rely on generating functions but otherwise the methods for the two cases are different. We stress that the calculations do not rely on Theorem \ref{thm:almostlocal} so we do not need to assume that the trees are critical. 

\subsubsection{Galton--Watson trees with bounded degrees}
We now restrict our attention to the case of Galton--Watson trees with degrees restricted to lie in a finite set $D$. In other words, the degrees are bounded above by some constant. By general principles of generating functions, we know that we can calculate the mean by studying the first derivative of $T(x,t)$ with respect to $t$. Likewise, we can find the variance by studying the second derivative. Using the fact that $\log\vert\aut T\vert$ is an additive functional, we start with the following expression, which was derived for general additive functionals in \cite{MR3318048}
\begin{equation}
\label{eq:AddFunc1Derivative}
    T_t(x, 0) = \frac{xT_x(x, 0)}{T(x, 0)}H(x) , 
\end{equation}
where $H(x) = \sum w(T) f(T) x^{\vert T\vert}$. We already know the singular expansion for $T(x)$ from \eqref{eq:GWSingExp} and we can differentiate it termwise to obtain a singular expansion for $T_x(x)$. Thus, it is enough to study $H(x)$. We manipulate the function in the following way.
\begin{multline*}
    H(x) = \sum_T w(T) x^{\vert T\vert} \left(\sum_{i=1}^k \log(m_i!)\right) \\
    = \sum_T w(T) x^{\vert T\vert} \left(\sum_{\substack{B\textrm{ branch of }T, \\\textrm{up to isomorphism}}} \log(\textrm{mult}(B)!)\right) \\
    = \sum_B \sum_{m=1}^\infty \log(m!)\sum_{\substack{T:B\textrm{ }m\textrm{-fold }\textrm{branch}\\\textrm{ of }T\textrm{ up to iso.}}} w(T) x^{\vert T\vert} .
\end{multline*}
Note that the sum $B$ is over isomorphism classes (i.e., rooted Pólya trees).

Using the fact that $B$ occurs exactly $m$ times in $T$, we can rewrite the innermost sum as
\begin{multline*}
    x\sum_{k=m}^\infty w_k \binom{k}{m} (W(B)x^{\vert B\vert})^m \left(\sum_{T\neq B} w(T) x^{\vert T\vert}\right)^{k-m} \\
    = x\sum_{k=m}^\infty w_k (W(B)x^{\vert B\vert})^m \binom{k}{m} \left(T(x, 0)-W(B)x^{\vert B\vert}\right)^{k-m} .
\end{multline*}
This gives that $H(x)$ is equal to
\begin{multline*}
     x\sum_B \sum_{m=1}^\infty \frac{\log(m!)}{m!} (W(B)x^{\vert B\vert})^m  
    \sum_{k=m}^\infty w_k \frac{k!}{(k-m)!}  \left(T(x, 0)-W(B)x^{\vert B\vert}\right)^{k-m} \\
    = x\sum_B \sum_{m=1}^\infty \frac{\log(m!)}{m!} (W(B)x^{\vert B\vert})^m \Phi^{(m)}(T(x, 0)-W(B)x^{\vert B\vert})  ,
\end{multline*}
due to Taylor's theorem. As the degrees are bounded, there is some $k_0$ such that $w_k=0$ for $k\geq k_0$. Thus, $\Phi$ is a polynomial, and the inner sum (which is actually finite, as $\Phi^{(m)}(t)$ is eventually $0$) is a polynomial in $W(B)x^{\vert B\vert}$ and $T(x, 0)$. It follows that $H(x)$ can be expressed in the form
\begin{equation*}
H(x) = x \sum_{m=2}^M \sum_B \left(W(B)x^{\vert B\vert}\right)^{m} P_m(T(x,0)),
\end{equation*}
where $P_m$ is a polynomial. Note here that the sum starts at $m=2$ because $\log(1!) = 0$. Let us now consider the sum over $B$:
\begin{equation*}
\sum_B \left(W(B)x^{\vert B\vert}\right)^{m}.
\end{equation*}

In \cite{ProbTreeIso} it was shown that the probability that two Galton--Watson trees with bounded degrees are isomorphic, which is
\begin{equation*}
    p_n  = \frac{\sum_{\vert B\vert=n} W(B)^2}{\left(\sum_{\vert B\vert =n} W(B)\right)^2},
\end{equation*}
decays exponentially in $n$. Let $t_n = [x^n] T(x,0) = \sum_{\vert B \vert = n} W(B)$ be the total weight of all trees with $n$ vertices. We have, for every $m \geq 2$,
\begin{align*}
\sum_B \left(W(B)\vert x \vert^{\vert B\vert}\right)^{m} &= \sum_{n \geq 1} \sum_{\vert B \vert = n} W(B)^m \vert x \vert^{nm} \\
&\leq \sum_{n \geq 1} \left(\sum_{\vert B \vert = n} W(B) \right)^{m-2} \sum_{\vert B \vert = n} W(B)^2 \vert x \vert^{nm} \\
&= \sum_{n \geq 1} t_n^{m-2} p_n t_n^2 \vert x \vert^{nm} \\
&= \sum_{n \geq 1} p_n (t_n \vert x \vert^n)^m.
\end{align*}
As $p_n$ decays exponentially, this shows that the sum $\sum_B \left(W(B)\vert x \vert^{\vert B\vert}\right)^{m}$ has greater radius of convergence than $T(x,0)$, so it represents an analytic function in a disk around $0$ that contains the dominant singularity $\rho$ of $T(x,0)$ in its interior.

Thus, we can write $H(x) = G(x,T(x,0))$, where $G(x,t)$ is a polynomial in $t$ whose coefficients are functions of $x$ that are analytic in a larger region than $T(x,0)$. Thus the singular expansion for $T(x,0)$ carries over to a singular expansion for $H(x)$ around the dominant singularity $\rho$.
Applying this to \eqref{eq:AddFunc1Derivative}, we obtain a singular expansion for $T_t$. By the method of singularity analysis we find that the mean has the form $\mu n + O(1)$ for a constant $\mu$ given by
\begin{equation*}
    \mu = \frac{G(\rho,\tau)}{\tau}.
\end{equation*}
For given classes of Galton--Watson trees with bounded degrees, we can estimate the constant $\mu$ numerically by truncating the series 
\begin{equation*}
    G(x,t) = x\sum_B \sum_{m=1}^\infty \frac{\log(m!)}{m!} (W(B)x^{\vert B\vert})^m \Phi^{(m)}(t-W(B)x^{\vert B\vert})
\end{equation*}
and approximating its value at $(x,t)=(\rho,\tau)$. We can obtain asymptotics for the variance in a similar manner, but with lengthier calculations. We find that it has the form $\sigma^2 n + O(1)$ for some constant $\sigma^2$ that can also be computed.

We can, for example, estimate the moments for full binary trees (where every internal vertex has two children) and pruned binary trees (where every internal vertex has a left child, a right child, or both). Full binary trees have mean constant $\mu \approx 0.0939359$ and variance constant $\sigma^2 \approx 0.0252103$, and in the case of pruned binary trees we get $\mu \approx 0.0145850$ and $\sigma^2 \approx 0.0084835$. Both of these classes are closely related to the phylogenetic trees studied in \cite{MR2582703}, and the mean constants above agree with the one for phylogenetic trees after translating between the models.

\subsubsection{Labeled trees}

We now show how the constants $\mu$ and $\sigma^2$ in Theorem~\ref{thm:MainThmGW} can be computed for labeled trees with fairly good accuracy. To this end, we use the functional equation \eqref{eq:PolyaEqAut}. Note that we can rewrite it in terms of an analogously defined exponential generating function for rooted labeled trees. Set
\begin{equation*}R(x, t) = \sum_{T\in\mathcal{R}} \vert\aut{T}\vert^{t} \frac{x^{\vert T\vert}}{\vert T\vert!}, 
\end{equation*}
the sum now being over the set $\mathcal{R}$ of all rooted labeled trees. Since the number of distinct ways to label a P\'olya tree $T$ is $\vert T\vert!/\vert\aut{T}\vert$, we have the relation
\begin{equation*}
    R(x, t) = P(x, t-1), 
\end{equation*}
so the functional equation for P\'olya trees immediately translates to a functional equation for labeled trees:
\begin{equation}\label{eq:labeledtrees}
    R(x, t) = x \exp \left( \sum_{j=1}^\infty \frac{c(j, t-1)}{j}R(x^j, jt-j+1)\right) .
\end{equation}
When $t = 0$, one verifies easily (compare the calculations below for the derivative with respect to $t$) that $c(j, -1) = 0$ for $j > 1$ and $c(1, -1) = 1$, so the functional equation reduces to $R(x, 0) = x \exp(R(x, 0))$ as expected.

In order to determine the desired moments, we need to consider the derivatives with respect to $t$. To this end, note first that
\begin{multline*}
\sum_{j \geq 0} y^j \sum_{\lambda \vdash j} \prod_{k \geq 1} \frac{x_k^{\lambda_k}}{\lambda_k! k!^{\lambda_k}} = \prod_{k \geq 1} \sum_{\lambda_k \geq 0} \frac{x_k^{\lambda_k}y^{k \lambda_k}}{\lambda_k! k!^{\lambda_k}} \\
= \prod_{k \geq 1} \exp \Big(\frac{x_k y^k}{k!}\Big) = \exp \Big( \sum_{k \geq 1} \frac{x_k y^k}{k!} \Big).
\end{multline*}
Differentiating with respect to $x_m$ and plugging in $x_1 = x_2 = \cdots = x$ yields
\begin{equation*}
\sum_{j \geq 0} y^j \sum_{\lambda \vdash j} x^{\vert\lambda\vert-1} \lambda_m \prod_{k \geq 1} \frac{1}{\lambda_k! k!^{\lambda_k}} = \frac{y^m}{m!} \exp \Big( \sum_{k \geq 1} \frac{x y^k}{k!} \Big) = \frac{y^m}{m!} \exp(x(e^y-1)).
\end{equation*}
Consequently, 
\begin{equation*}
    \sum_{\substack{\lambda \vdash j \\ \vert\lambda\vert = r}} \lambda_m \prod_{k \geq 1} \frac{1}{\lambda_k! k!^{\lambda_k}} = [x^{r-1}y^j] \frac{y^m}{m!} \exp(x(e^y-1)) = [y^{j-m}] \frac{(e^y-1)^{r-1}}{(r-1)!m!}.
\end{equation*}
By definition, we have
\begin{equation*}
    \frac{d}{dt} \frac{c(j, t)}{j} = \sum_{\lambda\vdash j}  \frac{(-1)^{\vert\lambda\vert-1}}{\vert\lambda\vert}  \binom{\vert\lambda\vert}{\lambda_1, \lambda_2, \ldots} \left(\prod_{n=1}^{\infty} n!^{\lambda_n t}\right) \sum_{m=1}^{\infty} \lambda_m \log(m!), 
\end{equation*}
which therefore becomes
\begin{align*}
    \frac{d}{dt} &\frac{c(j, t)}{j} \Big\vert_{t=-1} \\
    &= \sum_{r=1}^{\infty} \sum_{m=1}^{\infty} (-1)^{r-1}(r-1)! \sum_{\substack{\lambda \vdash j \\ \vert\lambda\vert = r}} \lambda_m \log(m!) \prod_{k \geq 1} \frac{1}{\lambda_k! k!^{\lambda_k}} \\
    &= \sum_{r=1}^{\infty} \sum_{m=1}^{\infty} (-1)^{r-1}(r-1)! \log(m!) [y^{j-m}] \frac{(e^y-1)^{r-1}}{(r-1)!m!} \\
    &= \sum_{m=1}^{\infty} \frac{\log(m!)}{m!} [y^{j-m}] e^{-y} = \sum_{m=1}^j \frac{\log(m!)}{m!} \frac{(-1)^{j-m}}{(j-m)!} \\
    &= \frac{1}{j!} \sum_{m=1}^j (-1)^{j-m} \binom{j}{m} \log(m!) = \frac{1}{j!} \sum_{m=1}^j (-1)^{j-m} \binom{j-1}{m-1} \log(m).
\end{align*}

Let us write $d(j)$ for this expression. Differentiating~\eqref{eq:labeledtrees} with respect to $t$ and setting $t=0$, we get
\begin{align*}
    R_t(x, 0) &= x \exp\left( \sum_{j=1}^\infty \frac{c(j, -1)}{j}R(x^j, 1-j)\right) \\
    &\quad \times \sum_{j=1}^{\infty} \Big( c(j, -1) R_t(x^j, 1-j) + \frac{d}{dt} \frac{c(j, t)}{j} \Big\vert_{t=-1} R(x^j, 1-j) \Big) \\
    &= R(x, 0) \Big( R_t(x, 0) + \sum_{j=1}^{\infty} d(j) R(x^j, 1-j) \Big).
\end{align*}
This can be solved for $R_t(x, 0)$:
\begin{equation*}
    R_t(x, 0) = \frac{R(x, 0)}{1-R(x, 0)} \sum_{j=2}^{\infty} d(j) R(x^j, 1-j).
\end{equation*}
Here, we are using the fact that $d(1)=0$. Now note that $d(j)$ rapidly goes to $0$ due to the factor $j!$ in the denominator and that the functions $R(x^j, 1-j)$ are all analytic in a larger region than $R(x, 0)$. Therefore, we can directly apply singularity analysis, based on the well-known singular expansion
\begin{equation*}
    R(x, 0) = 1 - \sqrt{2(1-ex)} + \cdots
\end{equation*}
of $R(x, 0)$ at its singularity $\frac1{e}$, which yields
\begin{equation*}
    R_t(x, 0) \sim \frac{1}{\sqrt{2(1-ex)}} \sum_{j=2}^{\infty} d(j) R(e^{-j}, 1-j).
\end{equation*}
The infinite series converges rapidly, allowing for a fairly accurate numerical computation. The mean constant $\mu$ in this special case is found to be $\mu = 0.0522901\ldots$, and similar calculations for the second derivative yield the variance constant $\sigma^2 = 0.0394984\ldots$.

\section{The automorphism group of P\'olya trees}
\label{sec:RootedPolya}

Since Theorem~\ref{thm:almostlocal} is not available for P\'olya trees, we want to prove asymptotic normality by using generating functions and singularity analysis. Recall that we defined the bivariate generating function $P(x, t) = \sum_{T\in\mathcal{P}} e^{t\log\vert\aut{T}\vert} x^{\vert T\vert}$. We now let $\mathcal{B}(T)$ denote the set of root branches of a particular tree, and $\mathcal{B}_I(T)$ denote the set of unique root branches up to isomorphism. 
Observe that for P\'olya trees there is exactly one tree in every isomorphism class so it will not be necessary to introduce separate notation for such classes.

By considering only the terms corresponding to the star on $n$ vertices, for each $n$, we obtain 
\begin{equation*}
    \sum_n (n-1)!^t x^n .
\end{equation*}
This is not analytic for any choice of $t>0$ and, thus, neither is the original generating function. This is the main obstacle in proving asymptotic normality. To circumvent this problem, we will introduce a cut-off, ignoring the contribution of highly symmetric vertices. This is similar to the proof, in \cite{almostLocal_2020}, of Theorem~\ref{thm:almostlocal}, but there the cut-off is in terms of the size of the tree instead of symmetric vertices. We can then use the following approximation result to extend the result from the cut-off random variables to the full additive functional.

\begin{lemma}
\label{lemma:approx}
Let $(X_n)_{n\geq 1}$ and $(W_{n, N})_{n, N\geq 1}$ be sequences of centered random variables. If we have 
\begin{enumerate}
    \item $W_{n, N} \xrightarrow[]{d}_n W_N$ and $W_N \xrightarrow[]{d} W$ for some random variables $W, W_1, W_2, \ldots$, and
    \item $\Var(X_n - W_{n, N})\xrightarrow[N]{} 0$ uniformly in $n$, 
\end{enumerate}
then $X_n \xrightarrow[]{d} W$.
\end{lemma}
This result follows e.g.~from \cite[Theorem 4.28]{KallenbergFoundProb}. 
We will apply Lemma \ref{lemma:approx} to variables $X_n$ defined by
\begin{equation*}
    \frac{\log\vert\aut{\mathcal{P}_n}\vert - \E(\log\vert\aut{\mathcal{P}_n}\vert)}{\sqrt{n}} , 
\end{equation*}
and $W_{n, N}$ being the, similarly normalized, random variable for the additive functional $F^{\leq N}(T)$, defined by having the toll function:
\begin{equation*}
    f^{\leq N}(T) = \sum_{B\in\mathcal{B}_I(T)} I(\mathrm{mult}(B)\leq N) \log(\mathrm{mult}(B)!) .
\end{equation*}
We note that $F(T) - F^{\leq N}(T) = F^{>N}(T)$ for an additive functional defined by
\begin{equation*}
    f^{>N}(T) = \sum_{B\in\mathcal{B}_I(T)} I(\mathrm{mult}(B) > N) \log(\mathrm{mult}(B)!) , 
\end{equation*}
so that we will, in fact, be interested in $\Var(F^{>N}(T_n))$ for the second condition of Lemma \ref{lemma:approx}. By straightforward modifications of \eqref{eq:GenFuncPolya}, we can define the generating functions 
\begin{equation*}
P^{\leq N}(x, t) = \sum_{T\in\mathcal{P}} e^{tF^{\leq N}(T)}x^{\vert T\vert}
\end{equation*}
and
\begin{equation*}
P^{> N}(x, t) = \sum_{T\in\mathcal{P}} e^{tF^{> N}(T)}x^{\vert T\vert}
\end{equation*}
for the corresponding cut-off functionals.

\subsection{Mean and variance}
We can now derive moments for the additive functionals $F, F^{\leq N}, F^{>N}$ with the help of generating functions and singularity analysis. The calculations are essentially the same in all cases so, to simplify the exposition, we perform them only for $F$ and indicate in the end how the results differ.

Due to general principles of generating functions, studying the mean and variance corresponds to studying $P_t(x, 0)$ and $P_{tt}(x, 0)$. According to calculations for general additive functionals from \cite{MR3318048}, we can write
\begin{equation}
\label{eq:PolyaFirstDerivative1}
  P_t(x, 0) =   xP_x(x, 0) \frac{\sum_T f(T)x^{\vert T\vert} + P(x, 0)\sum_{k\geq 2} P_t(x^k, 0)}{P(x, 0)(1+\sum_{k\geq 2} x^kP_x(x^k, 0))}  ,
  \end{equation}
  and
\begin{multline}
\label{eq:PolyaSecondDerivative1}
  P_{tt}(x, 0)  =   \frac{xP_x(x, 0)}{P(x, 0)(1+\sum_{k\geq 2}x^k P_x(x^k, 0))} \Bigg( P(x, 0) \Big(\sum_{k\geq1} P_t(x^k, 0)\Big)^2 \\ 
   + P(x, 0)\sum_{k\geq2} k P_{tt}(x^k, 0) + \sum_{T} x^{\vert T\vert} f(T) (2F(T)-f(T))\Bigg) , 
\end{multline}
for the first and second derivative. To perform singularity analysis, we must first find singular expansions for these expressions. To this end, we study the sums involved in them separately. 

Recall that $\rho_p= 0.33832\ldots$ is the dominant singularity of $P(x) = P(x,0)$.
Using the facts that $\rho_p<1$ so that $\rho_p^m<\rho_p$ for $m\geq 2$ and that $\log\vert\aut{T}\vert = O(\vert T\vert\log\vert T\vert )$, we see that the derivatives involving higher powers of $x$ are analytic in a larger region than $P(x, 0)$. Now, note that we can rewrite
\begin{equation*}
    2F(T)-f(T) = 2 \sum_{B\in \mathcal{B}(T)} F(B) + f(T) , 
\end{equation*}
so that it is enough to study $\sum x^{\vert T\vert}f(T)\sum F(B)$ and $\sum x^{\vert T\vert}f(T)^2$, as well as $\sum x^{\vert T\vert}f(T)$. We will now show that we can factor each of these expressions as $P(x, 0)$ times some function that is analytic in a larger radius than $\rho_p$. For the sum in the expression for the mean, we have
\begin{align*}
\label{eq:polyaInfSum1}
    \sum_T x^{\vert T\vert} & f(T) \\
    &= \sum_T x^{\vert T\vert} \sum_{B\in\mathcal{B}_I(T)} \log(\mathrm{mult}(B)!) 
    = \sum_{B\in\mathcal{P}} \sum_{m=1}^\infty \log(m!) \sum_{T:\mathrm{mult}(B)=m} x^{\vert T\vert} \\
    &= \sum_{B\in\mathcal{P}} \sum_{m=1}^\infty \log(m!) x^{m\vert B\vert} (P(x, 0) - x^{\vert B\vert}P(x, 0)) \\ 
    &= P(x) \sum_{B\in\mathcal{P}} \sum_{m=1}^\infty \log(m!) x^{m\vert B\vert} (1 - x^{\vert B\vert}) 
    = P(x)\sum_B \sum_{m=2}^\infty  \log(m) x^{m\vert B\vert}, 
\end{align*}
where we note that $P(x, 0)-x^{\vert B\vert}P(x, 0)$ equals the generating function for P\'olya trees without $B$ as a root branch. 
By taking absolute values, we can now bound
\begin{equation*}
   \sum_B \sum_{m=2}^\infty   \log(m) \vert x\vert^{m\vert B\vert} = O(\sum_B \vert x\vert^{2\vert B\vert}),
\end{equation*}
as long as $\vert x\vert < 1$. The extra power of 2 means that the sum converges for $\vert x\vert<\sqrt{\rho_p}<1$, so by the Weierstrass $M$-test, we have analyticity in a larger region than for the original generating function $P(x)$.

For the sum involving $\sum F(B)$, we have
\begin{multline*}
    \sum_T x^{\vert T\vert} \left(\sum_{B\in\mathcal{B}_I(T)} \log(\mathrm{mult}(B)!)\right) \left(\sum_{B\in\mathcal{B}_I(T)} \mathrm{mult}(B) F(B)\right) \\
    = \sum_{B\in\mathcal{P}} F(B) \sum_{m=1}^\infty m\log(m!) \sum_{T:\mathrm{mult}(B)=m} x^{\vert T\vert} \\
    + \sum_{\substack{B_1, B_2\in\mathcal{P}:\\B_1\neq B_2}} \sum_{m_1, m_2\geq 1} \log(m_1!)m_2 F(B_2) \sum_{\substack{T:\mathrm{mult}(B_1)=m_1\\\mathrm{mult}(B_2)=m_2}} x^{\vert T\vert}.
\end{multline*}
Using the fact that $\sum_B F(B)x^{m\vert B\vert} = P_t(x^m, 0)$ and performing calculations similar to above, the first sum can be seen to be
\begin{equation*}
    P(x, 0)\sum_{m=2}^\infty \log(m!m^{m-1}) P_t(x^m, 0) , 
\end{equation*}
where the sum is analytic in a larger region than the original function. To deal with the other sum, we first rewrite
\begin{equation*}
    \sum_{\substack{T:\mathrm{mult}(B_1)=m_1\\\mathrm{mult}(B_2)=m_2}} x^{\vert T\vert} = P(x, 0) x^{m_1\vert B_1\vert}(1-x^{\vert B_1\vert}) x^{m_2\vert B_2\vert}(1-x^{\vert B_2\vert}) .
\end{equation*}
Then, we note that
\begin{multline*}
    \sum_{\substack{B_1:\\B_1\neq B_2}} F(B_1) \sum_{m_1=1}^\infty m_1x^{m_1\vert B_1\vert}(1-x^{\vert B_1\vert}) \\
    = \sum_{m_1=1}^\infty \sum_{\substack{B_1:\\B_1\neq B_2}} F(B_1)x^{m_1\vert B_1\vert} 
    =  \sum_{j=1}^\infty P_t(x^{j}, 0) - \sum_{j=1}^\infty F(B_2)x^{j\vert B_2\vert}.
\end{multline*}
These observations let us rewrite the larger sum as
\begin{multline*}
    P(x, 0)\left(\sum_{j=1}^\infty P_t(x^{j}, 0)\right)\sum_{m=1}^\infty \log(m!) x^{m\vert B\vert}(1-x^{\vert B\vert}) \\
    - P(x, 0) \sum_B F(B)\sum_{m=1}^\infty \log(m!) x^{m\vert B\vert}(1-x^{\vert B\vert})\sum_{j=1}^\infty x^{j\vert B\vert} .
\end{multline*}
The first of these two sums can now be dealt with using calculations identical to those performed earlier, and further simplifications for the second sum allow us to rewrite the whole expression as $P(x, 0)$ multiplied by
\begin{equation*}
     \Bigg(\left( P_t(x, 0) + \sum_{m=2}^\infty P_t(x^m, 0)\right) 
     \sum_B \sum_{m=2}^\infty  \log(m) x^{m\vert B\vert}  -\sum_m \log(m!)P_t(x^{m+1}, 0)\Bigg) .
\end{equation*}
The sum $\sum x^{\vert T\vert}f(T)^2$ can be dealt with using similar techniques and we conclude that we can rewrite \eqref{eq:PolyaFirstDerivative1} and \eqref{eq:PolyaSecondDerivative1} as
\begin{align}
\label{eq:PolyaDerivatives2}
  P_t(x, 0) &=   xP_x(x, 0) \frac{H(x) + \sum_{k\geq 2} P_t(x^k, 0)}{(1+\sum_{k\geq 2} x^kP_x(x^k, 0))}  \nonumber, \\
  P_{tt}(x, 0) &=   \frac{xP_x(x, 0)}{(1+\sum_{k\geq 2}x^k P_x(x^k, 0))} \Bigg( \Big(P_t(x, 0) + \sum_{k\geq2} P_t(x^k, 0)\Big)^2 \nonumber \\ 
  & \qquad + \sum_{k\geq2} k P_{tt}(x^k, 0) + 2(P_t(x, 0)H(x) + K(x)) + L(x)\Bigg) , 
\end{align}
for functions $H(x)$, $K(x)$ and $L(x)$ that are analytic in a larger region than $P(x, 0)$. This puts us in a situation where we can perform singularity analysis to find the moments. Numerical computations yield $\mu=0.1373423\ldots$ and $\sigma^2=0.1967696\ldots$.

If we instead consider $F^{\leq N}(T)$ or $F^{>N}(T)$, the extra indicator function introduced in the expression will carry trough the calculations and affect the indices in the sums. In the sums with index $m$ above, we will sum up to $m=N$ in the first case and sum from $m=N+1$ to infinity in the second. In particular, for $F^{>N}(T)$, the corresponding analytic functions $H^{>N}(x)$, $K^{>N}(x)$ and $L^{>N}(x)$ will converge to zero within their region of convergence, if we let $N\to\infty$.

\subsection{Asymptotic normality for $\log\vert\aut{\mathcal{P}_n}\vert$}
If we introduce a cut-off and study $F^{\leq N}$ instead of $\log\vert\aut{T}\vert$, we can perform calculations completely analogous to the ones we did for \eqref{eq:PolyaEqAut} to obtain the functional equation
\begin{equation}\label{eq:PN}
    P^{\leq N}(x, t) = x \exp \left( P^{\leq N}(x, t) + \sum_{j=2}^\infty \frac{c_N(j, t)}{j}P^{\leq N}(x^j, jt)\right) , 
\end{equation}
where we define
\begin{equation*}
    c_N(j, t) = j \sum_{\lambda\vdash j}  \frac{(-1)^{\vert\lambda\vert-1}}{\vert\lambda\vert}  \binom{\vert\lambda\vert}{\lambda_1, \lambda_2, \ldots} \left(\prod_{n=1}^{N} n!^{\lambda_n t}\right) .
\end{equation*}

Except for the root, every vertex in the tree occurs as the child of some other vertex. This implies that it contributes to exactly one of the terms 
\begin{equation*}
    I(\mathrm{mult}(B)\leq N) \log(\mathrm{mult}(B)!) ,
\end{equation*}
in the expansion of $F^{\leq N}(T)$. Thus, as a crude upper bound, each of the $n$ vertices contributes at most $\log N!$ to the total value of the additive functional. Therefore, we see that $F^{\leq N}(T)=O(n)$ and, if we restrict to $\vert t\vert<\delta$ for some suitable $\delta>0$, 
\begin{equation*}
\label{eq:PolyaFuncEq}
    G(x, y, t) := x \exp \left( y + \sum_{j=2}^\infty \frac{c_N(j, t)}{j}P^{\leq N}(x^j, jt)\right)
\end{equation*}
is analytic in a region containing $x=\rho_p$, $y=\tau$. Theorem 2.23 in \cite{DrmotaRandTrees} now gives asymptotic normality for $F^{\leq N}(T)$, i.e. $W_N \sim \mathrm{N}(0, \sigma^2_N)$ for some constant $\sigma^2_N$.

Note that 
\begin{equation*}
\Var(X_n-W_{n, N}) = \frac{\Var(F(\mathcal{P}_n) - F^{\leq N}(\mathcal{P}_n))}{n} .
\end{equation*}
Since  $F(T) - F^{\leq N}(T) = F^{>N}(T)$, we want to show that $\Var(F^{>N}(T_n))/n\to 0$ when $N\to\infty$ which leads us to study $P^{>N}_{tt}(x, t)$. The reasoning from the last section shows that coefficients in Taylor expansions of $H^{>N}(x)$, $K^{>N}(x)$ and $L^{>N}(x)$ around $x=\rho_p$ go to zero as $N\to\infty$. By dominated convergence, the same is true for the expressions 
\begin{equation*}
\sum P_t(x^k, 0) \textrm{ and } \sum k P_{tt}(x^k, 0), 
\end{equation*}
since all terms of $P_t$ and $P_{tt}$ involve powers of $F^{>N}(T)$ and these go to zero for any fixed tree as $N\to\infty$. By studying \eqref{eq:PolyaDerivatives2} (except with $P^{>N}_{tt}(x, t)$ instead of $P_{tt}(x, t)$) we see that all the coefficients in the singular expansion of $P^{>N}_{tt}(x, t)$ depend on these quantities. Therefore, the expansion must be of the type 
\begin{equation*}
    a_N \left(1-\frac{x}{\rho_p}\right)^{-3/2} + b_N \left(1-\frac{x}{\rho_p}\right)^{-1} + c_N \left(1-\frac{x}{\rho_p}\right)^{-1/2} + O_N(1) , 
\end{equation*}
where each coefficient, as well as the error, goes to zero with $N$.

Performing singularity analysis, where we also subtract $\E(F^{>N}(\mathcal{P}_n))^2$ to get the variance, and dividing by $n$, gives us that
\begin{equation*}
    \Var(X_n - W_{n, N}) = \gamma^2_N + O_N\left(\frac{1}{n}\right) , 
\end{equation*}
for some constant $\gamma_N$ that goes to $0$ as $N \to \infty$. Moreover, the $O$-term converges uniformly to zero. This implies that the variance of $X_n - W_{n, N}$ goes to zero, uniformly in $n$ so that the approximation lemma applies. Thus, we can conclude asymptotic normality for $\log\vert\aut{\mathcal{P}_n}\vert$ from the asymptotic normality of $F^{\leq N}(\mathcal{P}_n)$ and finish the proof.

\section{Automorphisms of unrooted trees}
\label{sec:Unrooted}
We show how to extend our results to unrooted versions of labeled trees and Pólya trees. Even though it is not clear what an unrooted version of a Galton--Watson tree is in general, some special cases can be dealt with using methods similar to the ones below, e.g. labeled unrooted binary trees.

\subsection{Unrooted labeled trees}
We can define unrooted labeled trees on the same probability space as rooted trees by taking a rooted tree and unrooting it. As there are exactly $n$ unique ways of rooting any labeled tree, this gives the uniform probability measure on unrooted trees, assuming that we started with the uniform measure on rooted trees. 

Now let $T$ be a rooted tree of size $n$ and $T_v$ be the tree rooted at the vertex~$v$. Note that $\aut{T_v}$ is the stabilizer 
of $v$ in $\aut{T}$. Thus, we have
\begin{equation*}
    1 \leq \frac{\vert \aut{T}\vert}{\vert \aut{T_v}\vert} 
    = \vert \textrm{Orbit of }v\vert \leq \vert T\vert , 
\end{equation*}
due to the orbit-stabilizer theorem \cite[Lemma 6.1]{MR2339282}. Taking logarithms and normalizing, we find that
\begin{equation}\label{eq:rooted-unrooted-ineq}
    0 \leq \frac{\log\vert \aut{T}\vert - \mu n}{\sqrt{n}} - \frac{\log\vert \aut{T_v}\vert - \mu n}{\sqrt{n}} \leq \frac{\log n}{\sqrt{n}} , 
\end{equation}
with $\mu$ being the mean constant for rooted labeled trees from Theorem \ref{thm:MainThmGW}.
If we let $X_n=\frac{\log\vert \aut{T}\vert - \mu n}{\sqrt{n}}$ and likewise $Y_n = \frac{\log\vert \aut{T_v}\vert - \mu n}{\sqrt{n}}$ for rooted trees, then we see that we have almost sure convergence of $X_n - Y_n$ to 0, and thus also convergence in probability. Slutsky's theorem together with the result for rooted trees now lets us conclude that
\begin{equation*}
    X_n = X_n-Y_n + Y_n \xrightarrow{d} N(0, \sigma^2) , 
\end{equation*}
with $\sigma^2$ also coming from the theorem for rooted trees.

\subsection{Unrooted P\'olya trees}
We derive an analog of equation \eqref{eq:ClassicalUnrooted} that takes the size of the automorphism group into account.
Let us first recall that a centroid of a tree is a vertex with the property that none of the components obtained by removing it contains more than half of the vertices. It is a classical result going back to Jordan~\cite{MR1579443} (see also e.g.~\cite[Ex. 6.21a]{MR2321240}) that every tree has either a unique centroid (which we then call a central vertex) or two centroids, connected by an edge (called a central edge). Centroid vertices are also characterized by the property that the sum of the distances to all other vertices is minimized.

A central edge that connects two isomorphic trees will be called a symmetry line and the term ``central edge'' will be reserved for edges between centroid vertices that are not symmetry lines. The difference between the automorphisms of rooted trees compared to unrooted trees is that in the latter case any automorphism must preserve edges but not necessarily the root. We now have a bijection between Pólya trees $\mathcal{P}$ and the union of unrooted trees $\mathcal{U}$ and pairs of Pólya trees $P_1, P_2 \in \mathcal{P}$ with $P_1\neq P_2$. Observe that for a rooted tree, there are four cases:
\begin{enumerate}
    \item The root is a central vertex. There is a bijection from such trees to unrooted trees with a central vertex and, furthermore, any automorphism must preserve a central vertex so that the two trees have the same group of automorphisms.
    \item \label{list:symmetry} The root is one endpoint of a symmetry line. We have a bijection between trees with a symmetry line where one of its endpoints is the root and unrooted trees with a symmetry line. We simply root the tree at one of the endpoints and note that we get the same rooted Pólya tree no matter which endpoint we choose. Any automorphism must preserve the central edge, but due to symmetry any automorphism of the rooted tree corresponds to two automorphisms of the unrooted version since we can map the endpoints of the symmetry line into each other. 
    \item The root is one endpoint of a central edge. First note that we have a bijection between unrooted trees with a central edge $\mathcal{U}_{ce}$ and pairs of rooted trees $\mathcal{P}_{ce}^p$ that, if joined by an edge at the roots, result in a tree with that edge as central. We now have a bijection between the union $\mathcal{U}_{ce}\cup\mathcal{P}_{ce}^p$ and rooted trees with a central edge where one of the two endpoints is the root. This can be seen by, in the former case, choosing one of the vertices of the central edge as the root. For the rooted trees in bijection with unrooted trees, we note that any automorphism of an unrooted tree must preserve the central edge, and as it is not a symmetry line this implies that it must fix the root. For rooted trees in bijection with pairs of rooted trees we note that the two trees must be different but have the same size implying that no additional symmetry can occur when joining them. In both cases, the size of the automorphism group of the rooted tree is the same size as its counterpart. 
    \item The root satisfies none of the above. Then one root branch contains strictly more than half of the vertices and the tree decomposes into an unordered pair of rooted trees, i.e., the large branch and the rest of the tree (including the root). As the trees have different sizes this makes the decomposition unique and the size of the automorphism group of the original tree is simply the product of the groups of the two subtrees.
\end{enumerate}
Let $U_c(x, t)$ be the generating function for unrooted trees with a central vertex and $U_e(x, t)$ be the generating function for unrooted trees with a central edge or symmetry line. Combining the observations from above, and translating it to the level of generating functions, we find that
\begin{equation*}
    P(x, t) = U_c(x, t) + U_e(x, t) - 2^tP(x^2, 2t) + P(x^2, 2t) + \frac{1}{2} P(x, t)^2 - \frac{1}{2}P(x^2, 2t)
\end{equation*}
where the two middle terms involving $P(x^2, 2t)$ are correction terms corresponding to point \ref{list:symmetry}. above and the last two terms count unordered pairs of distinct rooted trees. By noting that $U(x, t) = U_c(x, t) + U_e(x, t)$ and rearranging we get
\begin{equation*}
    U(x, t) = P(x, t)  - \frac{1}{2} P(x, t)^2 + \left(2^t - \frac{1}{2}\right)P(x^2, 2t)
\end{equation*}
which is enough to obtain moments for $\log\vert \aut{T}\vert$ and calculations show that the mean and variance constants are the same as for rooted trees.

To extend the results for rooted trees to a full central limit theorem we use the far-reaching result in \cite[Theorem 1.3]{MR3983790}. This theorem shows that the random unrooted tree $U_n$ on $n$ vertices is close to a tree $T_n$ obtained by identifying the roots of a rooted Pólya tree $\mathcal{P}_{K_n}$ (of random size $K_n$) and a tree $B_n$ of stochastically bounded size $\vert B_n\vert = n-K_n+1 = O_P(1)$. To be precise, the total variation distance between $U_n$ and $T_n$ is $O(e^{-cn})$ for a constant $c > 0$.

In other words, an unrooted tree essentially consists of a large rooted Pólya tree and something small. Thus, we have
\begin{equation}\label{eq:TVD}
    P\left(\frac{\log\vert\aut{U}_n\vert-\mu n }{\sqrt{n}}\leq a\right) = P\left(\frac{\log\vert\aut{T}_n\vert-\mu n}{\sqrt{n}}\leq a\right) + O(e^{-cn}).
\end{equation}
Moreover, if $\vert B_n \vert \leq M$ for some fixed $M$, then we have $\log \vert\aut{T_n}\vert = \log \vert\aut{\mathcal{P}_{K_n}}\vert + O(\log n)$ by the same argument that gave us~\eqref{eq:rooted-unrooted-ineq}, and consequently
\begin{align*}
    \frac{\log\vert\aut{T_n}\vert-\mu n }{\sqrt{n}} &= \frac{\log\vert\aut{\mathcal{P}_{K_n}}\vert-\mu n}{\sqrt{n}} + O \Big( \frac{\log n}{\sqrt{n}} \Big) \\
    &= \frac{\log\vert\aut{\mathcal{P}_{K_n}}\vert-\mu K_n}{\sqrt{K_n}} + O \Big( \frac{\log n}{\sqrt{n}} \Big).
\end{align*}
Since $\vert B_n \vert = n - K_n + 1$ is stochastically bounded, we see that
\begin{equation*}
\frac{\log\vert\aut{T_n}\vert-\mu n }{\sqrt{n}} - \frac{\log\vert\aut{\mathcal{P}_{K_n}}\vert-\mu K_n}{\sqrt{K_n}} \xrightarrow{p} 0.
\end{equation*}
So an application of Slutsky's theorem in combination with~\eqref{eq:TVD} and the results for rooted Pólya trees proves the central limit theorem for the size of the automorphism group in the case of unrooted trees. 

\section*{Acknowledgements}
This work was supported by the Knut and Alice Wallenberg Foundation.

The authors would like to thank Benedikt Stufler for helpful suggestions relating to Section \ref{sec:Unrooted}.

\nocite{label}

\bibliography{references}

\end{document}